\documentclass{amsart}

\usepackage[utf8]{inputenc}
\usepackage[shortlabels,inline]{enumitem} 

\usepackage{microtype}  
\usepackage{flushend}   
\usepackage{mathtools}  
\usepackage{booktabs}   
\usepackage{array}      
\usepackage[nocompress,nosort]{cite}

\usepackage{caption}    

\usepackage{float}
\usepackage[noend]{algpseudocode}
\newfloat{algo}{tb}{lop}
\floatname{algo}{Algorithm}
\usepackage[french,english]{babel}
\usepackage{caption}  
\usepackage{comment}  
\usepackage{amsmath,amsfonts}
\usepackage{xcolor}
\usepackage{url}
\usepackage{hyperref}
\usepackage{tikz}

\makeatletter
\def\@acmplainindent{0pt}
\def\@acmdefinitionindent{0pt}
\def\@proofindent{\noindent}
\makeatother

\newtheorem{theo}{Theorem}[section]
\newtheorem{lemma}[theo]{Lemma}
\newtheorem{prop}[theo]{Proposition}
\newtheorem{coro}[theo]{Corollary}
\theoremstyle{definition}
\newtheorem{remark}[theo]{Remark}
\newtheorem{example}[theo]{Example}

\newenvironment{algoenv}[3][\linewidth]{
\begin{minipage}{#1}%
\flushleft
\rule{\textwidth}{.08em}\vspace{-0.5\baselineskip}\smallskip
\begin{description}[noitemsep]
\item[Input] #2
\item[Output] #3
\end{description}
\vspace{-0.7\baselineskip}
\rule{\textwidth}{.05em}
\begin{algorithmic}
}{\end{algorithmic}
\vspace{-.5\baselineskip}
\rule{\textwidth}{.08em}
\end{minipage}}
\newcommand{\softO}{\tilde{{O}}}
\newcommand{\bigO}{{{O}}}
\newcommand{\poly}{\text{poly}}

\newcommand{\Resultant}{\operatorname{Res}}

\newcommand{\Z}{\mathbb Z}
\newcommand{\gfield}{{k}}             
\newcommand{\GField}[1]{\mathbb{F}_{#1}} 
\newcommand{\Kfield}{{K}}
\newcommand{\Lfield}{{L}}

\newcommand{\Hol}{\mathcal H}
\newcommand{\Mer}{\mathcal K}
\newcommand{\residue}{\operatorname{res}}

\newcommand{\FOp}{\operatorname{\sf F}} 
\newcommand{\AlgFOp}{\operatorname{\sf F}}
\newcommand{\SerFOp}{\operatorname{\sf F}}
\newcommand{\SOp}[1]{\operatorname{\sigma}_{#1}} 
\newcommand{\AlgSOp}[1]{\operatorname{\sf S}_{#1}} 
\newcommand{\eval}[1]{\operatorname{\sf eval}_{#1}} 

\newcommand{\basis}{\mathcal B}

\newcommand{\Hone}{\textbf{(H1)}}
\newcommand{\Htwo}{\textbf{(H2)}}

\title[Fast coefficient computation for algebraic power series in char. $>0$]{Fast coefficient computation for algebraic power series in positive characteristic}

\author{Alin Bostan}
\address{Inria, France}
\email{alin.bostan@inria.fr}
\author{Xavier Caruso}
\address{CNRS, France}
\email{xavier.caruso@normalesup.org}
\author{Gilles Christol}
\address{IMJ, France}
\email{christol.gilles@gmail.com}
\author{Philippe Dumas}
\address{Inria, France}
\email{philippe.dumas@inria.fr}

\begin{document}

\begin{abstract} We revisit Christol's theorem on algebraic power series in
positive characteristic and propose yet another proof for it. This new proof
combines several ingredients and advantages of existing proofs, which make it
very well-suited for algorithmic purposes. We apply the construction used in
the new proof to the design of a new efficient algorithm for computing the
$N$th coefficient of a given algebraic power series over a perfect field of
characteristic~$p$. It has several nice features: it is more general, more
natural and more efficient than previous algorithms. Not only the arithmetic
complexity of the new algorithm is linear in $\log N$ and quasi-linear in~$p$,
but its dependency with respect to the degree of the input is much smaller
than in the previously best algorithm. {Moreover, when the ground field is
finite, the new approach yields an even faster algorithm, whose bit complexity
is linear in $\log N$ and quasi-linear in~$\sqrt{p}$}.
\end{abstract}

\maketitle


\section{Introduction}
\label{BCCD:sect:introduction}
Given a perfect field $k$ of characteristic $p>0$, we address the
following question: how quickly can one compute the $N$th coefficient~$f_N$ of
an algebraic power series
\[
  f(t) = \sum_{n\geq 0} f_n t^n \in \gfield[[t]],
\]
where $N$ is assumed to be a large positive integer?
This question was recognized as a very important one in complexity
theory, as well as in various applications to algorithmic number theory:
Atkin-Swinnerton-Dyer congruences, integer factorization,
discrete logarithm and point-counting~\cite{ChuChu90,BoGaSc07}.

As such, the question is rather vague; both the data structure and the
computation model have to be stated more precisely.
The algebraic series~$f$ will be specified in $\gfield[[t]]$ as some root of a
polynomial $E(t,y)$ in $\gfield[t,y]$, of degree $d = \deg_y E \geq 1$ and of
height $h = \deg_t E$. To do this specification unequivocally, we will make
several assumptions. First, we assume that~$E$ is \emph{separable}, that is
$E$ and its derivative~$E_y = {\partial E}/{\partial y}$ are coprime
in~$\gfield(t)[y]$. 
Second, we assume that~$E$ is
\emph{irreducible}\footnote{The first assumption
is not always implied by the second one, as exemplified by $E = y^p - t \in
\GField{p}[t,y]$, and in general by any irreducible polynomial $E$ in
$\gfield[t,y^p]$.} in~$\gfield(t)[y]$. 
Note that both assumptions are satisfied if $E$ is assumed to be the minimal polynomial of~$f$
and that irreducibility implies separability as soon as we know that $E$ has
at least one root in $k[[t]]$.
The polynomial $E$ might have several roots in $\gfield[[t]]$. In order 
to specify uniquely its root~$f$, we further assume that we 
are given a nonnegative integer $\rho$ together with 
$f_0, \ldots, f_{{2\rho}}$ in $\gfield$ such that 
\begin{align*}
E(t, f_0 + f_1 t + \dotsb + f_{2\rho} t^{2\rho}) 
& \equiv 0 \pmod{t^{2\rho+1}}, \\
E_y(t, f_0 + f_1 t + \dotsb + f_{\rho} t^{\rho}) 
& \not \equiv 0 \pmod{t^{\rho+1}}.
\end{align*}
In other words, the data structure used to represent~$f$ is the polynomial 
$E$ together with the initial coefficients~$f_0, \ldots, f_{2\rho}$. 
(Actually $\rho{+}1$ coefficients are enough to 
ensure the uniqueness of $f$. However $2\rho{+}1$ coefficients are needed to 
ensure its existence; for this reason, we will always assume the
coefficients of $f$ are given up to index $2\rho$.)
We observe that it is always possible to choose $\rho$ less than or
equal to the $t$-adic valuation of the $y$-resultant of $E$ and $E_y$, 
hence \emph{a fortiori} $\rho \leq (2d{-}1)h$.

Under these assumptions, the classical Newton iteration~\cite{KuTr78} 
allows the computation of the first $N$ coefficients of~$f$ in 
quasi-linear complexity $\softO(N)$. Here, and in the whole article 
(with the notable exception of Section~\ref{BCCD:sect:recurrences}), the 
algorithmic cost is measured by counting the number of basic arithmetic 
operations $(+,-,\times,\div)$ and applications of the Frobenius map ($x\mapsto 
x^{p}$) and of its inverse ($x \mapsto x^{1\hspace{-0.2ex}/\hspace{-0.2ex}p}$) 
in the ground field $k$. The soft-O notation $\softO( \cdot)$ indicates that 
polylogarithmic factors in the argument are omitted.
Newton's iteration thus provides a quasi-optimal algorithm to 
compute~$f_0,\ldots,f_N$. A natural and important question is whether 
faster alternatives exist for computing the coefficient~$f_N$ alone.

With the exception of the rational case ($d = 1$), where the $N$th coefficient
can be computed in complexity~$\bigO(\log N)$ by binary
powering~\cite{Fiduccia85}, the most efficient algorithm currently known to
compute~$f_N$ in characteristic~$0$ has complexity~$\softO(\sqrt N)$~\cite{ChudnovskyChudnovsky88}. It relies on baby step / giant step
techniques, combined with fast multipoint evaluation.

Surprisingly, in positive characteristic~$p$, a radically different approach
leads to a spectacular complexity drop to~$\bigO(\log N)$. However, the big-O
term hides a (potentially exponential) dependency in~$p$. The good behavior of
this estimate with respect to the index~$N$ results from two facts. First, if
the index~$N$ is written in radix $p$ as $(N_{\ell-1} \ldots N_1
N_0)_{p}$, then the coefficient~$f_N$ is given by the simple formula
\begin{equation}\label{BCCD:eq:basic-formula-fN}
f_N = [(\AlgSOp{N_{\ell-1}}\dotsb\AlgSOp{N_1}\AlgSOp{N_0} f)(0)]^{p^{\ell}},
\end{equation}
where the~$\AlgSOp{r}$ ($0 \leq r < p$) are the \emph{section} 
operators defined by
\begin{equation}\label{BCCD:eq:local-section}
	\AlgSOp{r} \sum_{n\geq 0} g_n t^n = \sum_{n\geq 0} g_{pn + r}^{1\hspace{-0.2ex}/\hspace{-0.2ex}p} t^n. 
\end{equation} 
Note that for the finite field~$\GField{p}$ the exponents~$p^{\ell}$ in~\eqref{BCCD:eq:basic-formula-fN}  and~$1/p$ in~\eqref{BCCD:eq:local-section} are useless, since the Frobenius map $x \mapsto x^p$ is the identity map in this case. 

Second, by Christol's theorem~\cite{Christol79,ChKaMeRa80,Harase88}, the
coefficient sequence of an algebraic power series $f$ over a perfect field
$\gfield$ of characteristic~$p>0$ is \emph{$p$-automatic}: this means that
$f$ generates a finite-dimensional
$k$-vector space under the action of the section operators. Consequently, with respect to a fixed $k$-basis of this
vector space, one can express $f$ as a column vector~$C$, the section
operators~$\AlgSOp{r}$ as square matrices~$A_r$ ($0\leq r <p$), and the
evaluation at~$0$ as a row vector~$R$.
Formula~\eqref{BCCD:eq:basic-formula-fN} then becomes
\begin{equation}\label{BCCD:eq:matrix-formula-fN} 
	f_N = [RA_{N_{\ell-1}}\dotsb A_{N_1}A_{N_0}C]^{p^{\ell}}. 
\end{equation} 
Since~$\ell$ is about~$\log N$, and since the size of the matrices $A_r$ does
not depend on $N$, formula \eqref{BCCD:eq:matrix-formula-fN} yields an
algorithm of complexity~$\bigO(\log N)$. This observation (for any
{$p$-automatic} 
sequence) is due to Allouche and Shallit~\cite[Cor. 4.5]{AlSh92}.
However, this last assertion hides the need to first find the linear
representation $(R, (A_r)_{0\leq r<p}, C)$. As shown 
in~\cite[Ex.~5]{BoChDu16}, already in the case of a finite prime field,
translating the $p$-automaticity in terms of linear algebra
yields matrices $A_r$ whose size can be about $d^2h p^{2d}$. Thus,
their precomputation has a huge impact on the cost with respect to the 
prime number~$p$.

In the particular case of a prime field $\gfield = \GField{p}$, and under the
assumption $E_y(0,f_0) \neq 0$, this was improved in~\cite{BoChDu16} by
building on an idea originally introduced by Christol in~\cite{Christol79}:
one can compute $f_N$ in complexity $\softO((h+d)^5hp) + O((h+d)^2
h^2 \log N)$. So far, this was the best complexity result for
this task.

\medskip
\noindent{\bf Contributions.}
We further improve the complexity result from~\cite{BoChDu16} down to 
$\softO(d^2 hp + d^\omega h) + O(d^2h^2 \log N)$ 
(Theorem~\ref{thm:Nth-via-Hermite-Pade}, 
Section~\ref{BCCD:sect:faster-algorithm}).
Here $\omega$ is the exponent of matrix multiplication.
In the case where $\gfield$ is a finite field, we propose an even faster 
algorithm, with bit complexity linear in $\log N$ and quasi-linear 
in~$\sqrt{p}$ (Theorem~\ref{theo:Nth-via-recurrence}, 
Section~\ref{BCCD:sect:recurrences}). It is obtained by blending 
the approach in Section~\ref{BCCD:sect:faster-algorithm} with 
ideas and techniques imported from the characteristic zero case~\cite{ChudnovskyChudnovsky88}.
All these successive algorithmic 
improvements are consequences of our main theoretical result 
(Theorem~\ref{BCCD:thm:stability}, Section~\ref{BCCD:sect:key-theorem}), 
which can be thought of as an effective version of Christol's Theorem
(and in particular reproves it).

\section{Effective version of Christol's theorem}

We keep the notation of the introduction. Christol's theorem is stated
as follows.

\begin{theo}[Christol]
\label{BCCD:thm:Christol}
Let~$f(t)$ in~$\gfield[[t]]$ be a formal power series that is algebraic 
over~$\gfield(t)$, where~$\gfield$ is a perfect field with positive 
characteristic.
Then there exists a finite-dimensional $\gfield$-vector space 
containing~$f(t)$ and stable by the section operators.
\end{theo}

\noindent
The aim of this section is to state and to prove an effective version
of Theorem~\ref{BCCD:thm:Christol}, on which our forthcoming 
algorithms will be built.
Our approach follows the initial treatment by
Christol~\cite{Christol79}, which is based on Furstenberg's theorem 
\cite[Thm.~2]{Furstenberg67}. For the application we have in mind, it
turns out that the initial version of Furstenberg's theorem will be
inefficient; hence we will first need to strengthen it, considering
residues around the moving point~$f(t)$ instead of residues at~$0$.
Another new input we shall use is a globalization argument 
allowing us to compare section operators at $0$ and at~$f(t)$.
This argument is formalized throught Frobenius operators and is 
closely related to the Cartier operator used in a beautiful geometric 
proof of Christol's theorem due to Deligne~\cite{Deligne84} and 
Speyer~\cite{Speyer10}, and further studied by Bridy~\cite{Bridy17}.

\subsection{Frobenius and sections}
\label{BCCD:sect:frobenius-sections}

Recall that the ground field~$\gfield$ is assumed to be a perfect field 
of prime characteristic~$p$, for example a finite field~$\GField{q}$, 
where $q=p^s$. Let $\Kfield = \gfield(t)$ be the field of rational 
functions over~$\gfield$ and let $\Lfield = \Kfield[y]/(E)$.

Since $\gfield$ is a perfect field, the Frobenius endomorphism $\AlgFOp: k
\rightarrow k$ defined by $x \mapsto x^p$ is an automorphism of $\gfield$. It
extends to a ring homomorphism, still denoted by~$\AlgFOp$,
from~$\Lfield[t^{1/p}]$ to~$\Lfield$ which raises an element
of~$\Lfield[t^{1/p}] = \Lfield^{1/p}$ to the power~$p$. 
This homomorphism is an isomorphism
and its inverse writes
\begin{equation}\label{BCCD:eq:global-section}
  \AlgFOp^{-1} = \sum_{r=0}^{p-1} t^{r/p} \AlgSOp{r},
\end{equation}
where each~$\AlgSOp{r}$, with $0 \leq r < p$, maps~$\Lfield$ onto itself. 

The use in~\eqref{BCCD:eq:global-section} of the same notation as in 
Formula~\eqref{BCCD:eq:local-section}
is not a mere coincidence.
The algebraic series~$f$ provides an embedding of~$\Lfield$ into the 
field of Laurent series~$\gfield((t))$, which is the evaluation of an 
element $P(y)$ of $\Lfield$ at the point~$y=f(t)$. We will 
call~$\eval{f} : L \rightarrow \gfield((t))$ the corresponding map, 
which sends $P(y)$ to $P(f(t))$.
The Frobenius operator extends from $L$ to~$\gfield((t))$, and the same 
holds for the sections~$\AlgSOp{r}$ ($0 \leq r < p$).
These extensions are exactly those of
Eq.~\eqref{BCCD:eq:local-section}. 
The~$\AlgSOp{r}$'s in Eq.~\eqref{BCCD:eq:global-section} then appear as 
global variants of the~$\AlgSOp{r}$'s in 
Eq.~\eqref{BCCD:eq:local-section}. Moreover, global and local operators 
are compatible, in the sense that they satisfy
\begin{equation}
  \AlgFOp\circ\eval{f} = \eval{f}\circ\AlgFOp, \quad 
  \AlgSOp{r}\circ\eval{f}  = \eval{f}\circ\AlgSOp{r}. 
\end{equation}

As for rational functions, the Frobenius operator and the section
operators induce, respectively, a ring isomorphism~$\FOp$
from~$\Kfield[t^{1/p}]$ onto~$\Kfield$ and maps~$\SOp{r}$ ($0 \leq r < p$)
from~$\Kfield$ onto~$\Kfield$ such that
$\AlgFOp^{-1} = \sum_{r=0}^{p-1} t^{r/p} \SOp{r}$.
The operators~$\AlgFOp$ and~$\AlgSOp{r}$ ($0 \leq r < p$) are not
$\Kfield$-linear but only $\gfield$-linear. More precisely, for any~$\lambda$
in~$\Kfield[t^{1/p}]$, $\mu$ in~$\Kfield$, and~$z$ in~$\Lfield$,
\begin{equation}\label{BCCD:eq:lack-of-k(t)-linearity}
  \AlgFOp(\lambda z)  = \FOp(\lambda) \AlgFOp(z) \quad\text{and}\quad
  \AlgSOp{r}(\mu z)  = \sum_{s=0}^{p-1} t^{\lfloor \frac{r+s}{p} \rfloor}
  \! 
  \SOp{s}(\mu) \AlgSOp{r-s}(z).
\end{equation}
In other words both $\AlgFOp$ and $\AlgFOp^{-1}$ are actually semi-linear.

\subsection{The key theorem}
\label{BCCD:sect:key-theorem}

Let ${\gfield[t,y]_{< h,< d}}$ be the set of polynomials $P\in\gfield[t,y]$
such that $\deg_t P < h$ and $\deg_y P < d$.
\begin{theo}\label{BCCD:thm:stability}
For $P \in {\gfield[t,y]_{< h,< d}}$ and for $0\leq r < p$, 
there exists a (unique) polynomial $Q$ in $\gfield[t,y]_{< h,< d}$ such that
\begin{equation}\label{BCCD:eq:stability}
  \AlgSOp{r}\left(\frac{P}{E_y}\right) \equiv \frac{Q}{E_y} \pmod E.
\end{equation}
\end{theo}

The rest of this subsection is devoted to the proof of 
Theorem~\ref{BCCD:thm:stability}. Although mainly algebraic, the proof 
is based on the rather analytic remark that any algebraic function in 
$k(t)[f]$ can be obtained as the residue at $T=f$ of some rational 
function in $k(t,T)$ (see Lemma~\ref{BCCD:lem:residue}).  This idea was 
already used in Furstenberg~\cite{Furstenberg67}, whose work has been 
inspiring for us.
The main new insight of our proof is the following: we replace several 
small branches around zero by a single branch around a moving 
point. In order to make the argument work, we shall need further to 
relate the behavior of the section operators around $0$ and around the 
aforementioned moving point. This is where the reinterpretation of the 
$\AlgSOp r$'s in terms of Frobenius operators will be useful.

\smallskip

We consider the ring $\Hol = k((t))[[T]]$ of power series over 
$k((t))$. Its fraction field is the field $\Mer = k((t))((T))$ of 
Laurent series over $k((t))$.
There is an embedding $\gfield((t))[y] \to \Hol$ taking a
polynomial in $y$ to its Taylor expansion around $f$. Formally, it
is simply obtained by mapping the variable $y$ to $f{+}T$. It extends 
to a field extension $\gfield((t))(y) \to \Mer$. We will often write 
$P(t,f{+}T)$ for the image of $P(t,y) \in \gfield((t))(y)$ in $\Mer$.
The field $\Mer$ is moreover endowed with a \emph{residue map} $\residue : 
\Mer \to k((t))$, defined by $\residue \big(\sum_{i=v}^\infty a_i T^i\big)  = a_{-1}$ (by convention, 
$a_{-1} = 0$ if $v > -1$). It is clearly $\gfield((t))$-linear.

\begin{lemma}
\label{BCCD:lem:residue}
For any polynomial $P \in \gfield((t))[y]$, the following equality holds:
$$\residue\left(\frac{P(t,f{+}T)}{E(t,f{+}T)}\right) = \frac{P(t,f)}{E_y(t,f)}.$$
\end{lemma}

\begin{proof}
Since $f$ is a simple root of $E$, the series $E(t,f{+}T)$ has a simple 
zero at $T=0$. This means that it can be written $E(t,f{+}T) = T \cdot 
q(T)$ with $q \in \Hol$, $q(0) \neq 0$.
Taking the logarithmic derivative with respect to~$T$ gives
\[
\frac{E_y(t,f{+}T) }{ E(t,f{+}T) } = \frac{ 1 }{ T } + \frac{ q'(T) }{ q(T) },
\]
akin to~\cite[Formula~(15), p.~276]{Furstenberg67},
from which we derive
\[
\frac{P(t,f{+}T)}{E(t,f{+}T)} = \frac {g(T)} T +
g(T) \, \frac{q'(T)}{q(T)},
\]
where  $g(T) = P(t,f{+}T) / E_y(t,f{+}T)$.
Since $E_y(t,f{+}T)$ does not vanish at $T=0$, the series $g(T)$ has 
no pole at $0$. Therefore, the residue of $g(T) / T$ is nothing
but $g(0)$. Besides the residue of the second summand 
$g(T)\, q'(T) / q(T)$ vanishes. All in all, the residue of 
$P(t,f{+}T)/E(t,f{+}T)$
is $g(0) + 0 = P(t,f)/E_y(t,f)$.
\end{proof}

We now introduce analogues of section operators over $\Mer$.
For this, we first observe that the Frobenius operator $x \mapsto
x^p$ defines an isomorphism $\SerFOp: \Mer[t^{1/p}, T^{1/p}] \to \Mer$.
Moreover $\Mer[t^{1/p}, T^{1/p}]$ is a field extension of 
$\Mer$ of degree $p^2$. A basis of $\Mer[t^{1/p}, T^{1/p}]$
over $\Mer$ is of course $(t^{r/p} \, T^{s/p})_{0 \leq r,s < p}$, but it
will be more convenient for our purposes to use a different one. It
is given by the next lemma.

\begin{lemma}
\label{BCCD:lem:basisMer}
The family 
$(t^{r/p} \, (f{+}T)^{s/p})_{0 \leq r,s < p}$ is a basis
of $\Mer[t^{1/p}, T^{1/p}]$ over~$\Mer$.
\end{lemma}

\begin{proof}
For simplicity, we set $y = f{+}T \in \Mer$.
We have:
$$\big(\begin{matrix}
1 & y^{1/p} & \cdots & y^{(p-1)/p}
\end{matrix}\big) = 
\big(\begin{matrix}
1 & T^{1/p} & \cdots & T^{(p-1)/p}
\end{matrix}\big) \cdot U$$
where $U$ is the square matrix whose $(i,j)$ entry (for $0 \leq
i,j < p$) is $\binom j i \, f^{i/p}$.
In particular, $U$ is upper triangular and all its diagonal entries
are equal to $1$. Thus $U$ is invertible and the conclusion follows.
\end{proof}

For $r$ and $s$ in $\{0, 1, \ldots, p{-}1\}$, we define 
the section operators $\AlgSOp{r,s} : \Mer \to \Mer$ by
$$\SerFOp^{-1} = \sum_{r=0}^{p-1} \sum_{s=0}^{p-1} 
t^{r/p} (f{+}T)^{s/p} \: \AlgSOp{r,s}.$$
(These operations look like those used in~\cite[\S3.2]{BoChDu16}, but
they are not exactly the same.)
Clearly $\AlgSOp{r,0}$ extends the operator $\AlgSOp{r} : \gfield((t)) 
\to \gfield((t))$ defined by Eq.~\eqref{BCCD:eq:local-section} and 
$\AlgSOp{r,s}(g_1^p g_2) = g_1 \AlgSOp{r,s}(g_2)$ for all $g_1,g_2
\in \Mer$. We 
observe moreover that the $\AlgSOp{r,s}$'s stabilize the subrings 
$\gfield((t))[y]$ and $\gfield[t,y]$, since $y$ corresponds to $f{+}T$.

\begin{prop}
\label{BCCD:prop:commResSec}
The following commutation relation holds over $\Mer$:
$$\AlgSOp r \circ\:\residue = \residue \circ \AlgSOp{r,p{-}1}.$$
\end{prop}

\begin{proof}
Let us write $g \in \Mer$ as $g = \sum_{i=v}^\infty a_i T^i$ 
with $v \in \Z$ and $a_i \in \gfield((t))$ for all $i \geq v$. Its
image under $\SerFOp^{-1}$ can be expressed in two different ways
as follows:
\[
\SerFOp^{-1}(g) 
 = \sum_{i=v}^\infty \AlgFOp^{-1}(a_i) \: T^{i/p}
 = \sum_{r=0}^{p-1} \sum_{s=0}^{p-1} t^{r/p} (f{+}T)^{s/p} \: \AlgSOp{r,s}(g).
\]
We identify the coefficient in $T^{-1/p}$. For doing so, we observe 
that the terms obtained with $s < p-1$ do not contribute, while the
contribution of the term $t^{r/p} (f{+}T)^{(p-1)/p} \: \AlgSOp{r,p-1}(g)$ 
is the residue of $t^{r/p} \: \AlgSOp{r,p-1}(g)$. We then get
$$\AlgFOp^{-1} (a_{-1})
= \sum_{r=0}^{p-1} \residue \circ \AlgSOp{r,p{-}1}(g) \cdot t^{r/p}.$$
Going back to the definition of $\AlgSOp r$, we derive 
$\AlgSOp r (a_{-1}) = \residue \circ \AlgSOp{r,p{-}1}(g)$, from which the
lemma follows.
\end{proof}

\begin{proof}[Proof of Theorem~\ref{BCCD:thm:stability}]
Let $P \in \gfield[t,y]$ and $0\leq r < p$.
We set $Q = \AlgSOp{r,p-1}(P E^{p-1}) \in \gfield[t,y]$. 
Combining Lemma \ref{BCCD:lem:residue} and Proposition
\ref{BCCD:prop:commResSec}, we derive the following 
equalities:
\begin{align*}
\AlgSOp r \left(\frac{P(t,f)}{E_y(t,f)}\right)
& = \AlgSOp r \circ\: \residue\left(\frac{P(t,f{+}T)}{E(t,f{+}T)}\right) \\
& = \residue \circ \AlgSOp{r,p-1} \left(\frac{P(t,f{+}T)}{E(t,f{+}T)}\right)
= \residue \left(\frac{Q(t,f{+}T)}{E(t,f{+}T)}\right)
  = \frac{Q(t,f)}{E_y(t,f)}
\end{align*}
(compare with~\cite[\S3.2]{BoChDu16}).
The stability of $\gfield[t,y]/E(t,y)$ under $\AlgSOp r$ follows using the
fact that $E$ is the minimal polynomial of $f$ over $\Kfield = \gfield(t)$.
If we know in addition that $P$ lies in $\gfield[t,y]_{< h, <d}$ then $P 
\, E^{p-1}$ is in $\gfield[t,y]_{< ph, \leq p(d-1)}$ and, therefore, $Q$ 
falls in $\gfield[t,y]_{<h, <d}$ as well.
Theorem~\ref{BCCD:thm:stability} is proved.
\end{proof}

\begin{remark}\label{rem:refined-degree-bound}
It is possible to slightly vary the bounds on the degree and the height,
and to derive this way other stability statements.
For example, starting from a polynomial $P(t,y)$ with $\deg_t P \leq h$
and $\deg_y P \leq d$, we have:
\[
  \AlgSOp {r} \frac{P(t,f)}{E_y(t,f)} =
  \frac{Q(t,f)}{E_y(t,f)}
\]
with $\deg_t Q \leq h$ and $\deg_y P < d$. Moreover $\deg_t Q < h$
provided that $r > 0$.

Another remark in this direction is the following: if~$P$ has degree at most~$d{-}2$, the 
section~$\AlgSOp{r,p-1}(P E^{p-1})$ has degree at most~$d{-}2$ for 
any $r \in \{0,1,\ldots, p{-}1\}$. Indeed, $P E^{p-1}$ has degree at
most~$pd - 2 < p (d{-}1) + p-1$. In other words, the 
subspace~$\gfield[t,y]_{< h,\leq d-2}$ is stable by the section 
operators~$\AlgSOp{r}$ ($0 \leq r < p$).
\end{remark}

\section{Application to algorithmics}

Theorem~\ref{BCCD:thm:stability} exhibits an easy-to-handle finite 
dimensional vector space which is stable under the section operators.
In this section, we derive from it two efficient algorithms that compute 
the $N$th term of $f$ in linear time in $\log N$.
The first is less efficient, but easier to understand; 
we present it mainly for pedagogical purposes.

\subsection{First algorithm: modular computations}
\label{BCCD:sect:fast-algorithm}

The first algorithm we will design follows rather straightforwardly from 
Theorem \ref{BCCD:thm:stability}. It consists of the following steps:
\begin{enumerate}[leftmargin=*]
\item we compute the matrix giving the 
action of the Frobenius $\FOp$ with respect to the ``modified 
monomial basis'' $\basis = ( {y^{j}}/{E_y})_{0 \leq j \leq d-1}$;
\item we deduce the matrix of $\FOp^{-1}$ with respect to $\basis$;
\item we extract from it the matrices of the section operators
$\AlgSOp r$;
\item we compute the $N$th coefficient of $f$ using
Formula~\eqref{BCCD:eq:basic-formula-fN}.
\end{enumerate}

\smallskip

Let us be a bit more precise (though we will not give full details 
because we will design in \S \ref{BCCD:sect:faster-algorithm} below an 
even faster algorithm). Let $M$ be the matrix of $\FOp$ in the 
basis~$\basis$; its $j$th column contains the coordinates of the vector 
$\FOp(\frac{y^{j}}{E_y}) = \frac {y^{pj}}{E_y^p}$ in the basis $\basis$, 
which are also the coordinates of $ {y^{pj}}/{E_y^{p-1}}$ in the 
monomial basis $(1, y, \ldots, y^{d-1})$. It is easily seen that 
the matrix of~$\AlgFOp^{-1}$ with respect to $\basis$ is 
$\FOp^{-1}(M^{-1})$, which is, by definition, the matrix obtained by 
applying~$\FOp^{-1}$ to each entry of~$M^{-1}$.

We now discuss the complexity of the computation of $M^{-1}$. Thanks to 
Theorem~\ref{BCCD:thm:stability} and Eq.~\eqref{BCCD:eq:global-section}, 
we know that its entries are polynomials of degree at most $h(p{-}1)$.
However, this bound is not 
valid for the entries of $M$. Indeed, in full generality, the latter
are rational fractions whose numerators and denominators have degrees 
of magnitude $dhp$. 
In order to save the extra factor $d$, we rely on modular techniques: 
we choose a polynomial $B$ of degree $h(p{-}1)+1$ and perform all 
computations modulo $B$. {To make the method work},  $B$ 
must be chosen in such a way that both $M$ and $M^{-1}$ make sense modulo $B$, 
\emph{i.e.} $B$ must be coprime with the denominators of the entries of 
$M$. The latter condition discards a small number of choices, so that 
a random polynomial $B$ will be convenient with high probability.

Using fast polynomial and matrix algorithms, 
the computation of $M$ modulo $B$ can 
be achieved within $\softO(d^2hp)$ operations in $k$, while the inversion 
of $M$ modulo $B$ requires $\softO(d^\omega h p)$ operations in $k$, where
$\omega \in [2,3]$ is the matrix multiplication exponent.
Since we count
an application of $\AlgFOp^{-1} : k \to k$ as a unique operation, the
cost of the first two steps is $\softO(d^\omega h p)$ as well. 
The third step is free as it only consists in reorganizing coefficients. 
As for the evaluation of the formula~\eqref{BCCD:eq:basic-formula-fN}, 
each application of $\AlgSOp r$ has a cost of $O(d^2h^2)$ operations in 
$k$. The total complexity of our algorithm is then $\softO(d^\omega hp) 
+ O(d^2 h^2 \log N)$ operations in $k$.

\begin{remark}
We do not need actually to apply the Frobenius inverse $\AlgFOp^{-1} : 
k \to k$ since, at the end of the computation, we raise the last 
intermediate result at the power $p^\ell$.
The complexity $\softO(d^\omega hp) + O(d^2 h^2 \log N)$ can then be
reached even if we do not count an application of $\AlgFOp^{-1}$ as
a single operation.
\end{remark}

\subsubsection*{A detailed example}

Consider $\gfield = \GField{5}$ and the polynomial
$$E = (t^4 + t + 1) y^4 + y^ 2 + y - t^4 \, \in \, \gfield[t,y].$$
It admits a unique root~$f$ in~$k[[t]]$ which is congruent to $0$
modulo $t$.

\begin{figure*}
\begin{center}
{\tiny
\noindent\hfill%
$\begin{array}{r@{\hspace{-15ex}}l}
 & M = \left(\begin{array}{rrrr}
1 + 2t^{4} + 4t^{5} + 3t^{6} + 2t^{7} + 2t^{8} + 2t^{12} + 4t^{13} + 3t^{14} + t^{15} + 2t^{16} & \cdots \smallskip \\
t + 2t^{2} + 3t^{3} + 4t^{5} + 4t^{6} + 3t^{7} + 3t^{8} + 3t^{9} + t^{10} + 4t^{11} + t^{13} + 2t^{15} + t^{16} & \cdots \smallskip \\
1 + 2t + 3t^{2} + 3t^{5} + 2t^{6} + 2t^{7} + t^{8} + t^{9} + 3t^{10} + t^{11} + 3t^{12} + 3t^{14} + 4t^{15} + 3t^{16} & \cdots \smallskip \\
0 & \cdots
\end{array}\right) \pmod{t^{17}} \bigskip \\
& M^{-1} =  \left(\begin{array}{rrrr}
  1 + 3t^{4} + t^{8} + t^{12} + t^{13} + t^{16} 
& t^{4} + 2t^{8} 
& t^{8} + t^{12} 
& t^{12} \smallskip \\
  2 + 2t + \cdots +
  t^{9} + 3t^{12}
& 3 + 2t + \cdots +
  2t^{13} + t^{16} 
& 3 + 4t + t^{5} + t^{8} 
& 1 + 4t^{4} + 3t^{5} + 2t^{9} + 2t^{12} \smallskip \\
  4 + 2t + \cdots +
  2t^{9} + 4t^{12}
& 4 + 2t + \cdots +
  + 2t^{9} + 4t^{12}
& 1 + 2t + \cdots + 
  3t^{13} + t^{16}
& 2 + 4t + \cdots +
  4t^{5} + 2t^{8} \smallskip \\
  0 
& 0 
& 0 
& 1 + 4t + \cdots +
  4t^{13} + t^{16}
\end{array}\right)
\end{array}$%
\hfill\null}

\caption{Frobenius and its inverse in the ``modified monomial basis''.}
\label{fig:Frobandinverse}
\end{center}
\end{figure*}

The matrix~$M$ of the Frobenius $\FOp$ with respect to the basis
$\basis = (\frac 1{E_y}, \frac y{E_y}, \frac {y^2}{E_y}, \frac {y^3}{E_y})$
writes~$D^{-1} \cdot \widetilde 
M$, where~$D$ and the largest entry of~$\widetilde M$ 
have degrees~$55$ and~$39$, respectively. However, by 
Theorem~\ref{BCCD:thm:stability}, we know that  $M^{-1}$ has 
polynomial entries of degree at most~$16$. Noticing that $0$ is not a 
root of the resultant of $E$ and $E_y$, we can compute $M$ and its 
inverse modulo $B(t) = t^{17}$. The result of this 
computation is displayed partly on Figure~\ref{fig:Frobandinverse}. We 
observe that the maximal degree of the entries of $M^{-1}$ is~$16$ and 
reaches our bound $h(p{-}1)$ (which is then tight for this example).
We furthermore observe that $M$ is block triangular, as expected after
Remark~\ref{rem:refined-degree-bound}.

Let us now compute the images of $y \in L$ under the section 
operators. For this, we write
$y = E_y^{-1} \cdot \big(4 t^4 + 2 y + 3 y^2\big)$ in $L$.
We then have to compute the product $M^{-1} \cdot 
(\begin{matrix} 4t^4 & 2 & 3 & 0 \end{matrix})^T$. As a result, we
obtain:
{\scriptsize $$\left(\begin{matrix} 
t^{4} + 4t^{8} + 2t^{12} + 4t^{16} + 4t^{17} + 4t^{20} \smallskip \\
t + 3t^{4} + 2t^{5} + t^{8} + 3t^{9} + 4t^{10} + 3t^{12} + 3t^{13} + 4t^{16} \smallskip \\
1 + 2t^{2} + 3t^{3} + 4t^{4} + t^{5} + t^{6} + 4t^{7} + 4t^{8} + 3t^{9} + 2t^{10} + 2t^{13} + 4t^{16} \smallskip \\
0
\end{matrix} \right).$$}%
Rearranging the terms, we finally find that
\begin{align*}
\AlgSOp 0(y) &= E_y^{-1} \cdot \big( 4t^4 + (2t + 4t^2) y + (1 + t + 2t^2) y^2 \big) \\
\AlgSOp 1(y) &= E_y^{-1} \cdot \big( 4t^3 + (1 + 4t^3) y + (t + 4t^3) y^2 \big) \\
\AlgSOp 2(y) &= E_y^{-1} \cdot \big( (2t^2 + 4t^3) + 3t^2 y + (2 + 4t) y^2 \big) \\
\AlgSOp 3(y) &= E_y^{-1} \cdot \big( 4t + (t + 3t^2) y + (3 + 4t + 2t^2) y^2 \big) \\
\AlgSOp 4(y) &= E_y^{-1} \cdot \big( 1 + (3 + 3t) y + (4 + 3t) y^2 \big). 
\end{align*}
To conclude this example, suppose that we want to compute the 
$70$th coefficient of~$f$. Applying Eq.~\eqref{BCCD:eq:basic-formula-fN}, 
we find that it is 
equal to the constant coefficient of $\AlgSOp{2}\AlgSOp{4}\AlgSOp{0} f$.
Therefore we have to compute $\AlgSOp{2}\AlgSOp{4}\AlgSOp{0} y$.
Repeating twice what we have done before, we end up with
$$\AlgSOp{2}\AlgSOp{4}\AlgSOp{0} y 
= E_y^{-1} \cdot \big( (2 + t^2) + (4 + 3t + 3t^3) y + (2 + 4t^2 + 2t^3)y^2\big).$$
Plugging $y = f$ in the above equality, we get 
$\AlgSOp{2}\AlgSOp{4}\AlgSOp{0} f = 2 + O(t)$, from which we conclude
that $f_{70} = 2$.

\begin{remark}
In the above example, only the constant coefficient of $f$ was 
needed to carry out the whole computation. This is related to the
fact that $E_y(f(t))$
has $t$-adic valuation $0$. More generally if $E_y(f(t))$
has $t$-adic valuation $\rho$, we will need the first $\rho{+}1$ 
coefficients of $f$ since the final division by $E_y$ will induce
a loss of $t$-adic precision of $\rho$ ``digits''. This does not
change the complexity bound, 
since $\rho \leq \deg_t \Resultant_y(E,E_y)
\in O(dh)$.
\end{remark}

\subsection{Second algorithm: Hermite--Padé approximation}
\label{BCCD:sect:faster-algorithm}

For obvious reasons related to the size of the computed objects, we 
cannot hope to achieve a complexity lower than linear with respect 
to~$p$ using the approach of Section~\ref{BCCD:sect:fast-algorithm}.
However, the exponent on~$d$ still can be improved. In order to achieve
this, we return to Theorem~\ref{BCCD:thm:stability}. The key idea is to 
leap efficiently from the polynomial~$P$ to the polynomial~$Q$ in 
Formula~\eqref{BCCD:eq:stability}.

Let $P=\sum_{i=0}^{d-1} a_i(t) y^i$ in $\gfield[t,y]_{< h,< d}$ and $0\leq r < p$.
By  Theorem~\ref{BCCD:thm:stability}, 
there exists $Q=\sum_{i=0}^{d-1} b_i(t) y^i$ in $\gfield[t,y]_{< h,< d}$ such that
${\AlgSOp{r}(P/E_y) \equiv Q/E_y \pmod E}$,
or, equivalently,
\begin{equation}\label{BCCD:eq:stability-bis}
 \AlgSOp{r}\left(\sum_{i=0}^{d-1} a_i(t) 
\frac{f(t)^i}{E_y(t,f(t))} \right)
= \sum_{j=0}^{d-1} b_j(t)  
\frac{f(t)^j}{E_y(t,f(t))}.
\end{equation}
The algorithmic question is to recover efficiently the $b_i$'s starting from
the $a_i$'s. Identifying coefficients in
Eq.~\eqref{BCCD:eq:stability-bis} yields a linear system over $\gfield$
in the coefficients of the unknown polynomials $b_i$. This system has $hd$
unknowns and an infinite number of linear equations. The point is that the following truncated version of Eq.~\eqref{BCCD:eq:stability-bis}
\begin{equation}\label{BCCD:eq:stability-bis-trunc}
 \AlgSOp{r}\left(\sum_{i=0}^{d-1} a_i(t) 
\frac{f(t)^i}{E_y(t,f(t))} \right)
\equiv \sum_{j=0}^{d-1} b_j(t)  
\frac{f(t)^j}{E_y(t,f(t))} \pmod{t^{2dh}}
\end{equation}
is sufficient to uniquely determine~$Q$.
This is a direct consequence of the following.

\begin{lemma}
If $Q$ in $\gfield[t,y]_{< h,< d}$ satisfies
$ \frac{Q}{E_y}(t,f(t)) \equiv 0 \pmod {t^{2dh}}$,
then $Q=0$.
\end{lemma}

\begin{proof}
The resultant $r(t)$ of $E(t,y)$ and $Q(t,y)$ with respect to $y$
is a polynomial of degree at most $d(h{-}1) + h(d{-}1)$. On the other 
hand, we have a Bézout relation 
$$E(t,y) \: u(t,y) + Q(t,y) \: v(t,y) = r(t),$$
where $u(t,y)$ and $v(t,y)$ are bivariate polynomials in $\gfield[t,y]$. 
By evaluating the previous equality at $y=f(t)$ it follows that 
$$r(t) \equiv Q(t,f(t)) \: v(t,f(t)) \equiv 0 \pmod{t^{2dh}}$$
holds in $\gfield((t))$, and therefore $r=0$. Thus $E$ and $Q$ have a non-trivial common factor; 
since $E$ is irreducible, it must divide~$Q$.
But  $\deg_y Q < \deg_y E$, so $Q=0$.
\end{proof}

Solving Eq.~\eqref{BCCD:eq:stability-bis-trunc} amounts to solving a
\emph{Hermite--Padé approximation} problem. In terms of linear algebra, it
translates into solving a linear system over $\gfield$ in the coefficients of
the unknown polynomials~$b_i$. This system has~$dh$ unknowns and $N=2dh$
linear equations. Moreover, it has a very special shape: it has a
quasi-Toeplitz structure, with displacement rank $\Delta=O(d)$. Therefore, it
can be solved using fast algorithms for structured
matrices~\cite{Pan01,BoJeSc08} in $\softO(\Delta^{\omega-1}N) =
\softO(d^{\omega}h)$ operations in~$\gfield$. These algorithms first compute a
(quasi)-inverse of the matrix encoding the homogenous part of the system,
using a compact data-structure called displacement generators (or, $\Sigma LU$
representation); then, they apply it to the vector encoding the inhomogeneous
part. The first step has complexity $\softO(\Delta^{\omega-1}N) =
\softO(d^{\omega}h)$, the second step has complexity $\softO(\Delta N) =
\softO(d^{2}h)$.

In our setting, we will need to solve $\log N$ systems of this type, each
corresponding to the current digit of~$N$ in radix~$p$. An important feature
is that these systems share the same homogeneous part, which only depends on
the coefficients of the power series $s_j(t) = \frac{f^j}{E_y(t,f(t))}$
occurring on the right-hand side of~\eqref{BCCD:eq:stability-bis-trunc}. Only
the inhomogeneous parts vary: they depend on the linear combination
$\sum_{i=0}^{d-1} a_i(t) s_i(t)$. Putting these facts together yields
Algorithm~\ref{BCCD:algo:Nth-coefficient-via-Hermite-Pade} and the following
complexity result.

\begin{algo}
  Algorithm \textbf{Nth coefficient \emph{via} Hermite-Padé.}
  \caption{\label{BCCD:algo:Nth-coefficient-via-Hermite-Pade}}
  \begin{algoenv}
    {A polynomial $E(t,y) = e_d(t)y^d + \dotsb + e_0(t)$ and a truncation $g = f_0 + \dotsb + \bigO(t^{\rho + 1})$ of a series~$f$ such that $E(t,g) = \bigO(t^{\rho + 1})$.}
    {The $N$th coefficient~$f_N$ of the series $f$.}
    \State 1. Precompute the first $2pdh$ coefficients of the series expansions~$s_j$ of~$f(t)^j/E_y(t,f)$, $0 \leq j < d$.
    \State 2. Precompute the quasi-inverse of the Toeplitz matrix corresponding to the Hermite--Padé approximation problem.
    \State 3. Expand $N = (N_{\ell-1}\ldots N_0)_p$ with respect to the radix~$p$.
    \State 4. Set $g = y \in L$ written as
      $\textstyle E_y^{-1} \cdot \big({-} d e_0 - (d{-}1) e_1 y - \cdots - e_{d-1} y^{d-1}\big).$
    \State 5. For $i = 0, 1, \ldots, \ell - 1$,
      \begin{enumerate}
        \item write~$g = P(t,f)/E_y(t,f)$ as a linear combination of the~$s_j$'s,
        \item compute the section $\AlgSOp{N_i}(g)$ at precision~$\bigO(t^{2dh})$,
        \item recover~$Q$ such that $\AlgSOp{N_i}(g) = Q/E_y$ by Hermite--Padé,
        \item redefine $g$ as $Q/E_y$
      \end{enumerate}
    \State 6. Replace $y$ by $\bar f(t)$ in $g$ and call $\bar g(t)$ the obtained result.
    \State 7. Expand $\bar g(t)$ at precision $\bigO(t)$.
    \State 8. Set $\bar g_0$ to the constant coefficient of~$\bar g(t)$.
    \State 9. Return $\bar g_0^{p^\ell}$.
  \end{algoenv}
\end{algo}

\begin{theo}
\label{thm:Nth-via-Hermite-Pade}
Let~$\gfield$ be a perfect field with characteristic~$p > 0$. 
Let~$E(t,y)$ be an irreducible polynomial in $\gfield[t,y]$ of
height~$h$ and degree~$d$. 
We assume that we are given a nonnegative integer $\rho$ and a 
polynomial $\bar f(t)$ such that $E(t,\bar f(t)) \equiv 0 
\pmod{t^{2\rho+1}}$ and $E_y(t,\bar f(t)) \not\equiv 0 
\pmod{t^{\rho+1}}$.

There there exists a unique series $f(t)$ congruent to $\bar f(t)$ 
modulo $t^{\rho+1}$ for which $E(t, f(t)) = 0$.
Moreover, Algorithm \ref{BCCD:algo:Nth-coefficient-via-Hermite-Pade} 
computes the $N$th 
coefficient of $f$ for a cost of $\softO(d^2 h p + d^\omega h) + 
\bigO(d^2 h^2 \log N)$ operations in~$k$.
\end{theo}

\begin{proof}
The first assertion is Hensel's Lemma~\cite[Th.~7.3]{Eisenbud1995}.

The precomputation of $s_j(t) = \frac{f(t)^j}{E_y(t,f(t))}$ modulo $t^{2dhp}$
for $0\leq j < d$ can be performed using Newton iteration, for a total cost of
$\softO(d^2hp)$ operations in $k$. As explained above, this is enough to set up the homogeneous part of the quasi-Toeplitz system; its inversion has cost
$\softO(d^{\omega}h)$.

Let us turn to the main body of the computation, which depends on the
index~$N$. For each $p$-digit $r=N_i$ of $N$, we first construct the
inhomogeneous part of the system. For this, we extract the coefficients of
$t^{pj+r}$ in $\sum_{i=0}^{d-1} a_i(t) s_i(t)$, for $0\leq j <d$, for a 
total cost of $O(d^2h^2)$ operations in $k$. We then apply the inverse 
of the system to it, for a cost of $O(d^2h^2)$ (using a naive matrix 
vector multiplication\footnote{One can actually achieve this step for
a cost of $\softO(d^2 h)$ operations in~$k$ using the quasi-Toeplitz 
structure; however this is not that useful since the cost of the previous 
step was already $O(d^2 h^2)$.}). This is done $\ell \approx \log N$ 
times. The other steps of the algorithm have negligible cost.
\end{proof}

\section{Improving the complexity with respect to $p$}
\label{BCCD:sect:recurrences}

As shown in Theorem \ref{thm:Nth-via-Hermite-Pade}, Algorithm 
\ref{BCCD:algo:Nth-coefficient-via-Hermite-Pade} has a nice complexity 
with respect to the parameters $d$, $h$ and $\log N$: it is polynomial 
with small exponents. However, the complexity with respect to $p$ is 
not that good, as it is exponential in $\log p$, which is the 
relevant parameter. Thus, when $p$ is large (say $> 10^5$), 
Algorithm~\ref{BCCD:algo:Nth-coefficient-via-Hermite-Pade} runs slowly 
and is no longer usable.

For this reason, it is important to improve the complexity with respect
to $p$. In this section, we introduce some ideas to achieve this.
More precisely, our aim is to design an algorithm whose complexity with 
respect to $p$ and $N$ is $\softO(\sqrt p) \cdot \log N$, and remains
polynomial in all other relevant parameters.
In the current state of knowledge, it seems difficult to decrease further
the exponent on $p$; indeed, the question addressed in this 
paper is related to other intensively studied questions (e.g., counting 
points \emph{via} $p$-adic cohomologies) for which the barrier 
$\softO(\sqrt p)$ has not been overcome yet.

\subsubsection*{Notations and assumptions}

We keep the notation of previous sections.
We make one additional hypothesis: \emph{the ground field $k$ is a finite field}. 
We assume that $k$ is 
represented as $(\Z/p\Z)[X]/\pi(X)$ where $\pi$ is an irreducible monic 
polynomial over $\Z/p\Z$ of degree $s$.
We choose a monic polynomial $\hat\pi \in \Z[X]$ of degree $s$ lifting 
$\pi$. We set $W = \Z_p[X]/\hat \pi(X)$ where $\Z_p$ is the ring of 
$p$-adic integers.

The algorithm we are going to design is not algebraic in the sense that it
does not only perform algebraic operations in the ground field~$k$, but will
sometimes work over $W$ (or, more exactly, over finite quotients of $W$). For
this reason, throughout this section, 
we will use bit complexity instead of algebraic complexity.

We use the notation $\poly(n)$ to indicate a quantity whose growth is at
most polynomial in $n$. The precise result we will prove reads as follows.

\begin{theo}
\label{theo:Nth-via-recurrence}
Under the assumptions of Theorem~\ref{thm:Nth-via-Hermite-Pade}
and the above extra assumptions, there exists an algorithm of bit
complexity $\poly(dh) \softO(s \sqrt p) \log N$ that computes the
$N$th coefficient of $f$.
\end{theo}

If $p$ is bounded by a (fixed) polynomial in $d$ and $h$, then Theorem 
\ref{theo:Nth-via-recurrence} has been proved already. In the sequel, we 
will then always assume that $p \gg d,h$.

\subsubsection*{Overview of the strategy}

We reuse the structure of Algorithm 
\ref{BCCD:algo:Nth-coefficient-via-Hermite-Pade} but speed up the 
computation of the $\AlgSOp {N_i}(g)$'s. Precisely, in Algorithm 
\ref{BCCD:algo:Nth-coefficient-via-Hermite-Pade}, the drawback was the 
computation of the $\frac{f^j}{E_y(t,f)}$'s or, almost equivalently, the 
computation of $g = \frac{P(t,f)}{E_y(t,f)}$ at sufficient precision. 
However, only a few (precisely $2dh$) coefficients of $g$ are needed, 
since we are only interested in one of its sections.
A classical method for avoiding this overhead is to find a (small) 
recurrence on the coefficients on $g = \sum_{n=0}^\infty g_i t^i$ of 
the form:
\begin{equation}
\label{eq:recurrence}
b_r(i) g_{i+r} + b_{r-1}(i) g_{i+r-1} + \cdots + b_1(i) g_{i+1} +
b_0(i) g_i = 0.
\end{equation}
We then unroll it using matrix factorials (for which fast 
algorithms are available in the literature {\cite{ChudnovskyChudnovsky88}}
Unrolling the recurrence is straightforward as soon as the 
leading coefficient $b_r(i)$ does not vanish. On the contrary, when 
$b_r(i) = 0$, the value of $g_{i+r}$ cannot be deduced from the previous 
ones. Unfortunately, it turns out that $b_r(i)$ does sometimes vanish in 
our setting.

We tackle this issue by lifting everything over $W$ and performing all 
computations over this ring. Divisions by $p$ then become possible but 
induce losses of precision. 
We then need to control the $p$-adic valuation of the denominators, that 
are the $p$-adic valuations of the $b_r(i)$'s. We cannot
expect to have a good control on them in full generality; even worse, we 
can build examples where $b_r(i)$ vanishes in $W$ for some $i$.
There exists nevertheless a good situation---the so-called \emph{ordinary
case}---where we can say a lot on the $b_r(i)$'s.
With this extra input, we are able to lead our strategy to its end.

The general case reduces to the ordinary one using a change of origin, 
\emph{i.e.} replacing $t$ by $u{+}\alpha$ for some $\alpha \in k$. This 
change of origin does not seem to be harmless \emph{a priori}. Indeed 
the Taylor expansion of $g$ around $\alpha$ (the one we shall compute) 
has in general nothing to do with the Taylor expansion of $g$ around $0$ 
(the one we are interested in). The sections are nevertheless closely 
related (see Proposition \ref{prop:SrtSru}). This ``miracle'' is quite 
similar to what we have already observed in Proposition 
\ref{BCCD:prop:commResSec} and again can be thought of as an avatar of 
the Cartier operator.

\subsection{From algebraic equations to recurrences}

We consider a bivariate polynomial $P(t,y) \in k[t,y]$ with 
$\deg_t P < h$ and $\deg_y P < d$. We fix moreover an integer~$r$ is the 
range $[0, p{-}1]$. Our aim is to compute 
$\AlgSOp r \!\big(\frac {P(t,f)}{E_y(t,f)}\big)$ at 
precision $O(t^{2dh})$.
Set $g = \frac{P(t,f)}{E_y(t,f)}$ and write $g = \sum_{i=0}^\infty
g_i t^i$. By definition $\AlgSOp r (g) = \sum_{j=0}^\infty g_{r+pj}^{1\hspace{-0.2ex}/\hspace{-0.2ex}p} t^j$, so that
we have to compute the coefficients $g_{r+pj}$ for $j < 2dh$.

We let $L$ be the leading coefficient of $E(t,y)$ and $R$ be the 
resultant of $E$ and $E_y$.
To begin with, we make the following assumption (which will be
relaxed in \S \ref{ssec:reducordinary}):

\medskip

\noindent
\Hone:
Both $L$ and $R$ have $t$-adic valuation $0$.

\medskip

As explained above, we now lift the situation over $W$.
We choose a polynomial $\hat E \in W[t,f]$ of bidegree $(h,d)$ lifting 
$E$. We define $\hat E_t = \frac{\partial \hat E}{\partial t}$, $\hat 
E_y = \frac{\partial \hat E}{\partial y}$. The assumption \Hone\ implies 
that the series $f$ lifts uniquely to a series $\hat f \in W[[t]]$ such 
that $\hat E(t,\hat f) = 0$.
We define $\hat L$ as the leading coefficient of $\hat E(t,y)$
and set $\hat R = \Resultant(\hat E, \hat E_y)$.
We introduce the ring $W_K = W[t, (\hat L\hat R)^{-1}]$.
By \Hone, $W_K$ embeds canonically into $W[[t]]$.
We pick a polynomial $\hat P \in W[t,y]$ lifting $P$ such that
$\deg_t \hat P < h$ and $\deg_y \hat P < d$. We set $\hat g = \hat P
(t,\hat f)$.

We now compute a linear differential equation satisfied by $\hat g$. For this, 
we observe that the derivation $\frac \partial{\partial t} : W[[t]] \to 
W[[t]]$ stabilizes the subring $W_L = W_K[\hat f]$.
Indeed from the relation $\hat E(t,\hat f) = 0$, we deduce that
$\frac{\partial \hat f}{\partial t} 
= - \frac{\hat E_t(t,\hat f)}{\hat E_y(t,\hat f)}$.
Thus $\frac{\partial \hat f}{\partial t} \in W_L$ because
$\hat E_y(t, \hat f)$ is invertible in $W_L$ thanks to \Hone. Using 
additivity and the Leibniz relation, we finally deduce that $\frac 
{\partial}{\partial t}$ takes $W_L$ to itself. In particular, all the 
successive derivatives of $\hat g$ lie in $W_L$.
On the other hand, we notice that $W_L$ is free of rank $d$ over $W_K$ 
with basis $(1, \hat f, \ldots, \hat f^{d-1})$.
Let $M$ be the $d \times d$ matrix whose $j$th column (for $0 \leq j < 
d$) contains the coordinates of $\frac{\partial^j\hat g}{\partial t^j}$ 
with respect to the above basis. Similarly let $C$ be the column vector 
whose entries are the coordinates of $\frac{\partial^d\hat g}{\partial
t^d}$. Let $\Delta_d = \det M$. We solve the system $M X = C$ using 
Cramer's formulae and find this way a linear differential equation of 
the form:
$$\Delta_d \frac {\partial^d \hat g}{\partial t^d} + 
\Delta_{d-1} \frac {\partial^{d-1} \hat g}{\partial t^{d-1}} + 
\cdots + \Delta_1 \frac {\partial \hat g}{\partial t} + 
\Delta_0 \hat g = 0,$$
where the other $\Delta_i$'s are defined as determinants as well.
In particular, they all lie in $W_K$. Multiplying by the appropriate
power of $\hat L \hat R$, we end up with a differential equation of
the form:
\begin{equation}
\label{eq:diffeqhatg}
\hat a_d \frac {\partial^d \hat g}{\partial t^d} + 
\hat a_{d-1} \frac {\partial^{d-1} \hat g}{\partial t^{d-1}} + 
\cdots + \hat a_1 \frac {\partial \hat g}{\partial t} + 
\hat a_0 \hat g = 0
\end{equation}
where the $\hat a_i$'s are polynomials in $t$.
We can even be more precise.
Indeed, following the above constructions, we find that all entries of 
$M$ and $C$ are rational functions whose degrees (of numerators and 
denominators) stay within $\poly(dh)$. We then deduce that the degrees 
of the $\hat \Delta_i$'s and $\hat a_i$'s are in $\poly(dh)$ as well. 
Furthermore, they can be computed for a cost of $\poly(dh)$ operations 
in $k$, that is $\poly(dh) \softO(s \log p)$ bit operations (recall that 
$s$ denotes the degree of $k$ over $\GField{p}$)

We write $\hat g = \sum_{i=0}^\infty \tilde g_i \frac{t^i}{i!}$.
The differential equation \eqref{eq:diffeqhatg} translates to a 
recurrence relation on the $\tilde g_i$'s of the form:
\begin{equation}
\label{eq:rechatg}
\forall n \geq r, \quad
\tilde b_0(n) \tilde g_n + \tilde b_{1}(n) \tilde g_{n-1} +
\tilde b_2(n) \tilde g_{n-2} + \cdots + \tilde b_r(n) \tilde g_{n-r} = 0
\end{equation}
where the $\tilde b_i$'s are polynomials in~$n$ over $W$ whose degrees 
are in $\poly(dh)$. Moreover $r$ is at most 
$d + \max_i \deg \hat a_i$. In particular, $r \in \poly(dh)$. 
Finally it is easy to write down explicitly $\tilde b_0$: it is 
the constant polynomial with value $\hat a_d(0)$.

\subsection{The ordinary case}
\label{ssec:ordinary}

In order to take advantage of Eq.~\eqref{eq:rechatg}, we make the 
following extra assumption, corresponding to the so-called 
\emph{ordinary case}:

\medskip

\noindent
\Htwo:
$\hat a_d(0)$ does not vanish modulo $p$.

\medskip

\noindent
Under \Htwo, $\tilde b_0(n) = \hat a_d(0)$ is invertible in $W$ and there is 
no obstruction to unrolling the recurrence \eqref{eq:rechatg}.
Let us be more precise.
We recall that we want to compute the values of $g_{r+pj}$ for $j$ up to 
$2dh$. Clearly $g_n$ is the reduction modulo $p$ of~$\frac{\tilde g_n}{n!}$. In order to get $g_{r+pj}$, we  need to 
compute $\tilde g_{r+pj}$ modulo $p^{v+1}$ where $v$ is the $p$-adic 
valuation of $(r+2dhp)!$. Under our assumption that $p$ is large enough compared
to $d$ and $h$, we get $v = 2dh$. We will then work over the finite
ring $W' = W/p^{2dh+1}W$.

We first compute the $r$ first coefficients of $\hat f$ modulo 
$p^{2dh+1}$ by solving the equation $\hat E(t,\hat f) = 0$ (using a 
Newton iteration for example).
Since $r \in \poly(dh)$, this computation can be achieved for a cost of 
$\poly(dh)$ operations in $W'$, that is $\poly(dh)\softO(s \log p)$ 
bit operations. We then build the companion matrix:
$$M(n) = \left( \begin{matrix}
& 1 \\
& & \ddots \\
& & & 1 \\
\frac{-\tilde b_r(n)}{\hat a_d(0)} & 
\frac{-\tilde b_{r-1}(n)}{\hat a_d(0)} & \cdots &
\frac{-\tilde b_1(n)}{\hat a_d(0)}
\end{matrix} \right) \in (W'[n])^{r \times r}.$$
Obviously,
$$\left( \begin{matrix} 
\tilde g_{n-r+1} & \tilde g_{n-r+2} & \cdots & \tilde g_n
\end{matrix} \right)^T =
M(n) \cdot M(n{-}1) \cdots M(r) \cdot
\left( \begin{matrix} 
\tilde g_0 & \tilde g_1 & \cdots & \tilde g_{r-1} 
\end{matrix} \right)^T,$$
and computing $\tilde g_n$ reduces to evaluating
the  matrix factorial $M(n) \cdot M(n{-}1) \cdots M(r)$. 
Using \cite{ChudnovskyChudnovsky88}, the latter can be
computed within $\poly(dh) \softO(\sqrt n)$ operations in $W'$, that
is $\poly(dh) \softO(\sqrt n \cdot s \log p)$ bit operations.
All in all, we find that the $g_{r+pj}$'s ($0 \leq j < 2dh$) can all 
be computed for a cost of $\poly(dh) \softO(s \sqrt p)$ bit
operations.

Plugging this input in Algorithm 
\ref{BCCD:algo:Nth-coefficient-via-Hermite-Pade}, we end up with an 
algorithm of bit complexity $\poly(dh) \softO(s \sqrt p) \log N$. 
Theorem \ref{theo:Nth-via-recurrence} is thus proved under the extra 
assumptions \Hone\ and \Htwo.

\subsection{Reduction to the ordinary case}
\label{ssec:reducordinary}

We finally explain how \Hone\ and \Htwo\ can be relaxed. 
The rough idea is to translate the origin at some point where these
two hypotheses hold simultaneously.

\subsubsection*{The case of complete vanishing}

Before proceeding, we need to deal with the case where the whole polynomial
$\hat a_d$ vanishes modulo $p$. This case is actually very special; this is
shown by the next lemma, whose proof relies on the fact 
that for a generic $g$,
the minimal-order (homogeneous) linear differential equation over $k(t)$ 
satisfied by~$g$ has order exactly~$d$~\cite{ChuChu86}.

\begin{lemma}
\label{lem:badg}
For a generic $g \in L = k(t)[y]/E(t,y)$, the reduction of $\hat a_d$ 
modulo $p$ does not vanish.
\end{lemma}


We say that an element $g \in L$ is \emph{good} if the corresponding 
$\hat a_d$ does not vanish modulo $p$. Lemma \ref{lem:badg} ensures
that goodness holds generically. It then holds with high probability
since we have assumed that the ground field $k$ has a large cardinality.
Consequently, even if we were unlucky and $g$ was not good, we can 
produce with high probability a decomposition $g = g_1 + g_2$ where 
$g_1$ and $g_2$ are both good (just by sampling $g_1$ at random). Since 
moreover the section $\AlgSOp r$ is additive, we can recover $\AlgSOp r 
(g)$ as $\AlgSOp r (g_1) + \AlgSOp r (g_2)$.

For this reason, in what follows we will assume safely that $g$ is good.

\subsubsection*{Change of origin}

Let $\hat \alpha \in W$ be such that $\hat L(\hat \alpha) \not \equiv 0 
\pmod p$, $\hat R(\hat \alpha) \not \equiv 0 \pmod p$, $\hat a_d(\hat 
\alpha) \not \equiv 0 \pmod p$. Such an element exists (since $k$ is 
assumed to be large) and can be found for a cost of $\poly(dh)$ 
operations in $k$ (e.g., by enumerating its elements).

We denote by $\alpha \in k$ the reduction of $\hat \alpha$ modulo $p$
and assume that $\alpha \neq 0$ (otherwise, we are in the ordinary
case).
We perform the change of variable $\tau_\alpha: t \mapsto u{+}\alpha$. 
Note that $\tau_\alpha$ induces isomorphisms $k(t) \to k(u)$ and
$k(t)[y]/E(t,y) \to k(u)[y]/E(u{-}\alpha,y)$. Furthermore,
the polynomial $E(\alpha,y) = 0$ has $d$ simple roots in an algebraic 
closure of $k$. Let $f_{\alpha,0}$ be one of them. By construction, 
$f_{\alpha,0}$ lies in a finite extension $\ell$ of $k$ of degree at 
most $d$. Moreover, by Hensel's Lemma, $f_{\alpha,0}$ lifts uniquely to 
a solution:
$$f_\alpha = f_{\alpha,0} + f_{\alpha,1} u + \cdots + f_{\alpha,i} u^i + 
\cdots \in \ell[[u]]$$
to the equation $E(u{-}\alpha, y) = 0$. We emphasize that the morphism
$k(t)[y]/E(t,y) \to k(u)[y]/E(u{-}\alpha,y)$ does \emph{not} extend to
a mapping $k((t)) \to \ell((u))$ sending $f$ to $f_\alpha$.
The next diagram summarizes the previous discussion:

\begin{center}
\begin{tikzpicture}[xscale=3,yscale=1.3]
\node (K) at (0,0) { $k(t)$ };
\node (Ka) at (1,0) { $k(u)$ };
\node (L) at (0,1) { $\displaystyle \frac{k(t)[y]}{E(t,y)}$ };
\node (La) at (1,1) { $\displaystyle \frac{k(u)[y]}{E(u{-}\alpha,y)}$ };
\node (M) at (0,2) { $k((t))$ };
\node (Ma) at (1,2) { $\ell((u))$ };
\draw (K)--(L)--(M);
\draw (Ka)--(La)--(Ma);
\draw[-latex] (K)--(Ka) node[midway, above] { $\tau_\alpha$ };
\draw[-latex] (L)--(La) node[midway, above] { $\tau_\alpha$ };
\draw[-latex,shorten >=1pt] (M) to [out=90,in=170,loop,looseness=4] (M);
\draw[-latex,shorten >=1pt] (Ma) to [out=90,in=10,loop,looseness=4] (Ma);
\node at (-0.2,2.6) { $\AlgSOp r$ };
\node at (1.2,2.6) { $\AlgSOp {r,u}$ };
\end{tikzpicture}
\end{center}

\noindent
Here $\AlgSOp r$ and $\AlgSOp {r,u}$ refer to the section operators 
defined in the usual way. We observe that they stabilize the subfields 
$\frac{k(t)[y]}{E(t,y)}$ and $\frac{k(u)[y]}{E(u{+}\alpha,y)}$,
respectively, since they can alternatively be defined by the relations:
\begin{equation}
\label{eq:defSrtSru}
\begin{array}{l@{\qquad}r@{\hspace{0.5ex}}l}
\text{over } \frac{k(t)[y]}{E(t,y)}\text{:} &
\AlgFOp^{-1} & = \sum_{r=0}^{p-1} t^{r/p} \AlgSOp r \medskip \\
\text{over } \frac{k(u)[y]}{E(u{-}\alpha,y)}\text{:} &
\AlgFOp^{-1} & = \sum_{r=0}^{p-1} u^{r/p} \AlgSOp {r,u} \\
\end{array}
\end{equation}
where $\AlgFOp$ is the Frobenius map 
(see also Eq.~\eqref{BCCD:eq:global-section}).

\begin{prop}
\label{prop:SrtSru}
The commutation 
$\AlgSOp {p-1,u} \circ\:\tau_\alpha = \tau_\alpha \circ \AlgSOp {p-1}$
holds over $\frac{k(t)[y]}{E(t,y)}$.
\end{prop}

\begin{proof}
Clearly $\tau_\alpha$ commutes with the Frobenius because it is a
ring homomorphism. From the relations \eqref{eq:defSrtSru}, we then
derive
$\sum_{r=0}^{p-1} u^{r/p} \: \AlgSOp {r,u} \circ\:\tau_\alpha =
  \sum_{r=0}^{p-1} (u{+}\alpha)^{r/p} \: \tau_\alpha \circ \AlgSOp r$.
Identifying the coefficients in $u^{\frac{p-1}p}$, we get the
announced result.
\end{proof}

We emphasize that the other section operators $\AlgSOp {r,\star}$ 
(with $r < p{-}1$) \emph{do not commute} with $\tau_\alpha$: the above phenomenon 
is specific to the index $p{-}1$.
However, we can relate $\AlgSOp r$ and $\AlgSOp {p-1,u}$ as follows.

\begin{coro}
\label{coro:SrtSru}
For all $g \in \frac{k(t)[y]}{E(t,y)}$, we have
$\AlgSOp r(g) = 
\tau_\alpha^{-1} \circ \AlgSOp {p-1,u} \circ\:\tau_\alpha (t^{p-1-r} g)$.
\end{coro}

\begin{proof}
This follows from Proposition \ref{prop:SrtSru}
and from
$\AlgSOp r(g) = \AlgSOp {p-1}(t^{p-1-r} g)$.
\end{proof}

\subsubsection*{A modified recurrence}

In order to use Corollary \ref{coro:SrtSru}, we need to check that
 $\tau_\alpha(t^{p-1-r} g)$ fits the ordinary case.
Recall the differential equation satisfied by~$\hat g$, 
$$\hat a_d \frac {\partial^d \hat g}{\partial t^d} + 
\hat a_{d-1} \frac {\partial^{d-1} \hat g}{\partial t^{d-1}} + 
\cdots + \hat a_1 \frac {\partial \hat g}{\partial t} + 
\hat a_0 \hat g = 0.$$
We set $r' = p-1-r$ and $\hat G = t^{r'} \hat g$. Applying Leibniz 
formula to $\hat g = t^{-r'} \hat G$, we get:
$$\frac{\partial^j \hat g}{\partial t^j} =
  \sum_{i=0}^j (-1)^i \binom{j}{i} r' (r'+1) \cdots (r'+i-1) t^{-r'-i} 
  \frac{\partial^j \hat G}{\partial t^{j-i}},$$
from which we derive the following differential equation satisfied
by $\hat G$:
$$\sum_{0 \leq i \leq j \leq d}
(-1)^i \hat a_j \binom j i r' (r'+1) \cdots (r'+i-1) t^{-r'-i}  
  \frac{\partial^{j-i} \hat G}{\partial t^{j-i}}.$$
Reorganizing the terms and multiplying by $t^{r'+d}$, we end up
with:
\begin{equation}
\label{eq:eqdiffG}
\sum_{j=0}^d \sum_{i=0}^{d-j}
(-1)^i \hat a_{i+j} \binom{i+j}{i} r' (r'+1) \cdots (r'+i-1) t^{d-i}
  \frac{\partial^j \hat G}{\partial t^j}.
\end{equation}
Set $W_{L,u} = W[u,y]/\hat E(u{+}\alpha,y)$ and define the ring
homomorphism
$\tau_{\hat \alpha} : W_L \to W_{L,u}$, $t \mapsto u{+}\hat\alpha$, 
$y \mapsto y$. Clearly $\tau_{\hat \alpha}$ lifts $\tau_\alpha$.
Applying $\tau_{\hat \alpha}$ to Eq.~\eqref{eq:eqdiffG} and noticing 
that $\frac{\partial}{\partial t} = \frac{\partial}{\partial u}$, we 
obtain:
$$\sum_{j=0}^d \sum_{i=0}^{d-j}
(-1)^i \hat \tau_{\hat\alpha}(a_{i+j}) \binom{i+j}{i} r' (r'+1) \cdots (r'+i-1) (u{+}\hat\alpha)^{d-i}
  \frac{\partial^j \tau_{\hat \alpha}(\hat G)}{\partial u^j}.$$ 

\subsubsection*{Conclusion}

The leading term of the latest differential equation (obtained only with 
$j = d$ and $i = 0$) is $\tau_{\hat\alpha}(\hat a_d) \: 
(u{+}\hat\alpha)^d$. Its value at $u = 0$ is then $\hat 
a_d(\hat\alpha)\:\hat\alpha^d$, which is not congruent to $0$ modulo $p$ by assumption. 
Moreover the other coefficients are polynomials in $u$ whose degrees 
stay within $\poly(dh)$.
Therefore, we can apply the techniques of \S \ref{ssec:ordinary} and 
compute $\AlgSOp {p-1,u} (\tau_\alpha G)$ at precision 
$O(u^{2dh})$ for a cost of $\poly(dh) \softO(s \sqrt p)$ bit 
operations. As explained in \S \ref{BCCD:sect:faster-algorithm}, we 
can reconstruct $\AlgSOp {p-1,u} (\tau_\alpha G)$ as an element of 
$k[u,y]/E(u{+}\alpha,y)$ for a cost of $\poly(dh)$ operations in $k$ 
using Hermite--Padé approximations.
Thanks to Corollary \ref{coro:SrtSru}, it now just remains to apply 
$\tau_\alpha^{-1}$ to get $\AlgSOp r (g)$. This last operation can
be performed for a cost of $\poly(dh)$ operations in $k$ as well.
All in all, we are able to compute $\AlgSOp r (g)$ for a total 
bit complexity of $\poly(dh) \softO(s \sqrt p)$. Repeating this
process $\log N$ times, we obtain the complexity announced 
in Theorem \ref{theo:Nth-via-recurrence}.

\smallskip
\bibliographystyle{abbrv}

\end{document}